\documentclass{amsart}
\usepackage{amsfonts}

\newtheorem{theorem}{Theorem}
\theoremstyle{plain}

\newtheorem{corollary}{Corollary}

\numberwithin{equation}{section}
\input{tcilatex}

\begin{document}
\title[$x^{2}-\left( a^{2}b^{2}-b\right) y^{2}=N$ and $x^{2}-\left(
a^{2}b^{2}-2b\right) y^{2}=N$ ]{The generalized Fibonacci and Lucas
solutions of the Pell equations $x^{2}-\left( a^{2}b^{2}-b\right) y^{2}=N$
and $x^{2}-\left( a^{2}b^{2}-2b\right) y^{2}=N$ }
\author{Bilge PEKER}
\address{Department of Mathematics Education, Ahmet Kelesoglu Education
Faculty, Necmettin Erbakan University, Konya, Turkey.}
\email{bilge.peker@yahoo.com}
\author{Hasan SENAY}
\address{Education Faculty, Mevlana University, Konya, Turkey}
\email{hsenay@mevlana.edu.tr}
\date{March 4, 2013}

\begin{abstract}
In this study, we find continued fraction expansion of $\sqrt{d}$ when $%
d=a^{2}b^{2}-b$ and $d=a^{2}b^{2}-2b$ where $a$ and $b$ are positive
integers. We consider the integer solutions of the Pell equations $%
x^{2}-\left( a^{2}b^{2}-b\right) y^{2}=N$ and $x^{2}-\left(
a^{2}b^{2}-2b\right) y^{2}=N$ when $N\in \left\{ \pm 1,\pm 4\right\} $. We
formulate the $n$-th solution $\left( x_{n},y_{n}\right) $ by using the
continued fraction expansion. We also formulate the $n$-th solution $\left(
x_{n},y_{n}\right) $ in terms of generalized Fibonacci and Lucas sequences.
\end{abstract}

\maketitle

\bigskip \textbf{1. Introduction and Preliminaries}

\bigskip The equation $x^{2}-dy^{2}=N,$ with given integers $d$ and $N,$
unknowns $x$ and $y$, is called as Pell equation. In the literature, there
are several methods for finding the integer solutions of Pell equation such
as the Lagrange-Matthews-Mollin algorithm, the cyclic method, Lagrange's
system of reductions, use of binary quadratic forms, etc.

If $d$ is negative, the equation can have only a finite number of solutions.
If $d$ is a perfect square, i.e. $d=a^{2},$ the equation reduces to $\left(
x-ay\right) \left( x+ay\right) =N$ and there is only a finite number of
solutions. If $d$ is a positive integer but not a perfect square, then
simple continued fractions are very useful. The simple continued fraction
expansion of $\sqrt{d}\ $has the form $\sqrt{d}=\left[ a_{0},\overline{%
a_{1},a_{2},a_{3},...,a_{m-1},2a_{0}}\right] $ with $a_{0}=\left[ \sqrt{d}%
\right] $. If the fundamental solution of $x^{2}-dy^{2}=1$ is $x=x_{1}$ and $%
y=y_{1}$, then all nontrivial solutions are given by $x=x_{n}$ and $y=y_{n}$%
, where $x_{n}+y_{n}\sqrt{d}=\left( x_{1}+y_{1}\sqrt{d}\right) ^{n}$. If a
single solution $\left( x,y\right) =\left( g,h\right) $ of the equation $%
x^{2}-dy^{2}=N$ is known, other solutions can be found. Let $\left(
r,s\right) $ be a solution of the unit form $x^{2}-dy^{2}=1$. Then $\left(
x,y\right) =\left( gr\pm dhs,gs\pm hr\right) $ are solutions of the equation 
$x^{2}-dy^{2}=N$.

Given a continued fraction expansion of $\sqrt{d}$, where all the $a_{i}$'s
are real and all except possibly $a_{0}$ are positive, define sequences $%
\left\{ p_{n}\right\} $ and $\left\{ q_{n}\right\} $ by $p_{-2}=0$, $%
p_{-1}=1 $, $p_{k}=a_{k}p_{k-1}+p_{k-2}$ and $q_{-2}=1$, $q_{-1}=0$, $%
q_{k}=a_{k}q_{k-1}+q_{k-2}$ for $k\geq 0$. Let $m$ be the length of the
period of continued fraction. Then the fundamental solution of $%
x^{2}-dy^{2}=1$ is%
\begin{equation*}
\left( x_{1},y_{1}\right) =\left\{ 
\begin{array}{c}
\left( p_{m-1},q_{m-1}\right) \\ 
\left( p_{2m-1},q_{2m-1}\right)%
\end{array}%
\right. 
\begin{array}{c}
\text{if }m\text{ is even} \\ 
\text{if }m\text{ is odd.}%
\end{array}%
\end{equation*}

\bigskip\ If the length of the period of continued fraction is even, then
the equation $x^{2}-dy^{2}=-1$ has no integer solutions. If $m$ is odd, the
fundamental solution of $x^{2}-dy^{2}=-1$ is given by $\left(
x_{1},y_{1}\right) =\left( p_{m-1},q_{m-1}\right) $ $\left[ 3\right] $.

Now let us give the following known theorems $\left[ 4\right] $ that will be
needed for the next section.

\begin{theorem}
Let $d\equiv 2\left( \func{mod}4\right) $ or $d\equiv 3\left( \func{mod}%
4\right) $. Then the equation $x^{2}-dy^{2}=-4$ has positive integer
solutions if and only if the equation $x^{2}-dy^{2}=-1$ has positive integer
solutions.
\end{theorem}

\begin{theorem}
Let $d\equiv 0\left( \func{mod}4\right) $. If fundamental solution of the
equation $x^{2}-\left( d/4\right) y^{2}=1$ is $x_{1}+y_{1}\sqrt{d/4}$, then
fundamental solution of the equation $x^{2}-dy^{2}=4$ is $\left(
2x_{1},y_{1}\right) $.
\end{theorem}

\begin{theorem}
Let $d\equiv 1\left( \func{mod}4\right) $ or $d\equiv 2\left( \func{mod}%
4\right) $ or $d\equiv 3\left( \func{mod}4\right) $. If fundamental solution
of the equation $x^{2}-dy^{2}=1$ is $x_{1}+y_{1}\sqrt{d}$, then fundamental
solution of the equation $x^{2}-dy^{2}=4$ is $\left( 2x_{1},2y_{1}\right) $.
\end{theorem}

In this study $\left[ 2\right] $, since generalized Fibonacci and Lucas
sequences related solutions of the forthcoming Pell equations are going to
be taken into consideration, let us briefly recall the generalized Fibonacci
sequences $\left\{ U_{n}\left( k,s\right) \right\} $ and Lucas sequences $%
\left\{ V_{n}\left( k,s\right) \right\} $. Let $k$ and $s$ be two non-zero
integers with $k^{2}+4s>0$. Generalized Fibonacci sequence is defined by%
\begin{equation*}
U_{0}\left( k,s\right) =0,U_{1}\left( k,s\right) =1
\end{equation*}%
and%
\begin{equation*}
U_{n+1}\left( k,s\right) =kU_{n}\left( k,s\right) +sU_{n-1}\left( k,s\right) 
\end{equation*}%
for $n\geq 1$. Generalized Lucas sequence is defined by%
\begin{equation*}
V_{0}\left( k,s\right) =2,V_{1}\left( k,s\right) =k
\end{equation*}%
and%
\begin{equation*}
V_{n+1}\left( k,s\right) =kV_{n}\left( k,s\right) +sV_{n-1}\left( k,s\right) 
\end{equation*}%
for $n\geq 1$. It is also well-known from the literature that generalized
Fibonacci and Lucas numbers have many interesting and significant
properties. Binet's formulas are probably the most important one among them.
For generalized Fibonacci and Lucas sequences, Binet's formulas are given by 
$U_{n}\left( k,s\right) =\frac{\alpha ^{n}-\beta ^{n}}{\alpha -\beta }$ and $%
V_{n}\left( k,s\right) =\alpha ^{n}+\beta ^{n}$ where $\alpha =\left( k+%
\sqrt{k^{2}+4s}\right) /2$ and $\beta =\left( k-\sqrt{k^{2}+4s}\right) /2$ $%
\left[ 5\right] $.

There are a large number of studies concerning Pell equation in the
literature. G\"{u}ney $\left[ 1\right] $ solved the Pell equations $%
x^{2}-\left( a^{2}b^{2}+2b\right) y^{2}=N$ when $N\in \left\{ \pm 1,\pm
4\right\} $.

In this study, we consider the integer solutions of the Pell equations 
\begin{equation*}
x^{2}-\left( a^{2}b^{2}-b\right) y^{2}=N\text{ and }x^{2}-\left(
a^{2}b^{2}-2b\right) y^{2}=N
\end{equation*}%
in terms of the generalized Fibonacci and Lucas numbers.

\bigskip

\textbf{2. The Pell equation }$x^{2}-\left( a^{2}b^{2}-b\right) y^{2}=N$%
\textbf{\ }

\bigskip Firstly, we consider the integer solutions of the Pell equation%
\begin{equation}
E:x^{2}-\left( a^{2}b^{2}-b\right) y^{2}=1.  \tag{1}
\end{equation}

\begin{theorem}
Let $E$ be the Pell equation in $\left( 1\right) $ and $a\geq 2$. Then the
followings hold:

$\left( \mathbf{i}\right) $ The continued fraction expansion of $\sqrt{%
a^{2}b^{2}-b}$ is 
\begin{equation*}
\sqrt{a^{2}b^{2}-b}=\left\{ 
\begin{array}{c}
\left[ a-1;\overline{1,2a-2}\right] \\ 
\left[ ab-1;\overline{1,2a-2,1,2ab-2}\right] \text{\ \ \ \ }%
\end{array}%
\right. ,%
\begin{array}{c}
\text{if }b=1 \\ 
\text{if }b>1\text{.}%
\end{array}%
\end{equation*}

$\left( \mathbf{ii}\right) $ The fundamental solution is%
\begin{equation*}
\left( x_{1},y_{1}\right) =\left\{ 
\begin{array}{c}
\left( a,1\right) \\ 
\left( 2a^{2}b-1,2a\right) \text{\ \ \ \ }%
\end{array}%
\right. ,%
\begin{array}{c}
\text{if }b=1 \\ 
\text{if }b>1\text{.}%
\end{array}%
\end{equation*}

$\left( \mathbf{iii}\right) $ The n-th solution $\left( x_{n},y_{n}\right) $
can be find by%
\begin{equation*}
\frac{x_{n}}{y_{n}}=\left\{ 
\begin{array}{c}
\left[ a-1;\left( 1,2a-2\right) _{n-1},1\right] \\ 
\left[ ab-1;\left( 1,2a-2,1,2ab-2\right) _{n-1},1\right] \text{\ \ \ \ }%
\end{array}%
\right. ,%
\begin{array}{c}
\text{if }b=1 \\ 
\text{if }b>1\text{.}%
\end{array}%
\end{equation*}%
where $\left( 1,2a-2\right) _{n-1}$ and $\left( 1,2a-2,1,2ab-2\right) _{n-1}$
mean that there are $n-1$ successive terms $\left( 1,2a-2\right) $ and $%
\left( 1,2a-2,1,2ab-2\right) $, respectively.\qquad \qquad \qquad \qquad
\qquad

\begin{proof}
$\left( \mathbf{i}\right) $ If $b=1$, then%
\begin{equation*}
\begin{array}{c}
\sqrt{a^{2}-1}=a-1+\left( \sqrt{a^{2}-1}-\left( a-1\right) \right) =a-1+%
\frac{2a-2}{\sqrt{a^{2}-1}+a-1}=a-1+\frac{1}{\frac{\sqrt{a^{2}-1}+a-1}{2a-2}}
\\ 
=a-1+\frac{1}{1+\frac{\sqrt{a^{2}-1}-a+1}{2a-2}}=a-1+\frac{1}{1+\frac{1}{%
\sqrt{a^{2}-1}+a-1}}\text{ \ \ \ \ \ \ \ \ \ \ \ \ \ \ \ \ \ \ \ \ \ \ } \\ 
=a-1+\frac{1}{1+\frac{1}{\sqrt{a^{2}-1}+a-1+a-1-\left( a-1\right) }}=a-1+%
\frac{1}{1+\frac{1}{2a-2+\left( \sqrt{a^{2}-1}-\left( a-1\right) \right) }}.%
\end{array}%
\end{equation*}

Therefore $\sqrt{a^{2}-1}=\left[ a-1;\overline{1,2a-2}\right] $.

If $b>1$, then%
\begin{equation*}
\begin{array}{c}
\sqrt{a^{2}b^{2}-b}=ab-1+\left( \sqrt{a^{2}b^{2}-b}-(ab-1)\right) \text{ \ \
\ \ \ \ \ \ \ \ \ \ \ \ \ \ \ \ \ \ \ \ \ \ \ \ \ \ \ \ \ \ \ \ \ \ \ \ \ \
\ } \\ 
\\ 
=ab-1+\frac{1}{\frac{\sqrt{a^{2}b^{2}-b}+(ab-1)}{2ab-b-1}}=ab-1+\frac{1}{1+%
\frac{\sqrt{a^{2}b^{2}-b}-(ab-b)}{2ab-b-1}}\text{ \ \ } \\ 
\\ 
\text{ \ \ \ \ \ }=ab-1+\frac{1}{1+\frac{1}{\frac{\sqrt{a^{2}b^{2}-b}+\left(
ab-b\right) }{b}}}=ab-1+\frac{1}{1+\frac{1}{2a-2+\frac{\sqrt{a^{2}b^{2}-b}%
-\left( ab-b\right) }{b}}} \\ 
\\ 
\text{ \ \ \ \ \ \ \ \ \ \ \ \ \ \ \ \ }=ab-1+\frac{1}{1+\frac{1}{2a-2+\frac{%
1}{\frac{\sqrt{a^{2}b^{2}-b}+\left( ab-b\right) }{2ab-b-1}}}}=ab-1+\frac{1}{%
1+\frac{1}{2a-2+\frac{1}{1+\frac{1}{\sqrt{a^{2}b^{2}-b}+\left( ab-1\right) }}%
}} \\ 
\\ 
\text{ \ \ \ \ \ \ \ \ \ \ \ \ \ \ \ \ \ \ }=ab-1+\frac{1}{1+\frac{1}{2a-2+%
\frac{1}{1+\frac{1}{\frac{2ab-b-1}{\sqrt{a^{2}b^{2}-b}-\left( ab-1\right) }}}%
}}=ab-1+\frac{1}{1+\frac{1}{2a-2+\frac{1}{1+\frac{1}{\sqrt{a^{2}b^{2}-b}%
+\left( ab-1\right) }}}} \\ 
\\ 
\text{ \ }=ab-1+\frac{1}{1+\frac{1}{2a-2+\frac{1}{1+\frac{1}{2ab-2+\left( 
\sqrt{a^{2}b^{2}-b}-\left( ab-1\right) \right) }}}}.\text{ \ \ \ \ \ \ \ \ \
\ \ \ \ \ \ \ \ \ \ }%
\end{array}%
\end{equation*}%
This completes the proof.

$\left( \mathbf{ii}\right) $ If $b=1$, then from $x_{n}=p_{nm-1}$ and $%
y_{n}=q_{nm-1}$, we obtain $x_{1}=p_{1}$ and $y_{1}=q_{1}$. Therefore we
must find $p_{1}$ and $q_{1}$. It is easily seen that $%
p_{1}=a_{1}p_{0}+p_{-1}=a$ and $q_{1}=$ $a_{1}q_{0}+q_{-1}=1$. That is, the
fundamental solution of $x^{2}-\left( a^{2}-1\right) y^{2}=1$ is $\left(
x_{1},y_{1}\right) =\left( a,1\right) $.

If $b>1$, then from $x_{n}=p_{nm-1}$ and $y_{n}=q_{nm-1}$, we obtain $%
x_{1}=p_{3}$ and $y_{1}=q_{3}$. Therefore we must find $p_{3}$ and $q_{3}$.
Now, we can find in a different way with the help of $3th$ convergent of $%
\sqrt{a^{2}b^{2}-b}$.

$\frac{p_{3}}{q_{3}}=\left[ ab-1;1,2a-2,1\right] =ab-1+\frac{1}{1+\frac{1}{%
2a-2+\frac{1}{1}}}=\frac{2a^{2}b-1}{2a}$. Therefore $\left(
x_{1},y_{1}\right) =\left( 2a^{2}b-1,2a\right) $.

$\left( \mathbf{iii}\right) $ If $b=1$, then it is known that $\left(
x_{1},y_{1}\right) =\left( a,1\right) $. For $n=1$, we obtain $\frac{x_{1}}{%
y_{1}}=\left[ a-1;1\right] =a-1+\frac{1}{1}=\frac{a}{1}$. Hence it is true
for $n=1$.

We assume that $\left( x_{n},y_{n}\right) $ is a solution of $x^{2}-\left(
a^{2}-1\right) y^{2}=1.$ That is, $\frac{x_{n}}{y_{n}}%
=[a-1;(1,2a-2)_{n-1},1] $\text{.}

Now we must show that it holds for $\left( x_{n+1},y_{n+1}\right) $.%
\begin{equation*}
\begin{array}{c}
\frac{x_{n+1}}{y_{n+1}}=a-1+\frac{1}{1+\frac{1}{2a-2+\frac{1}{1+\frac{1}{%
2a-2+\frac{1}{...2a-2+1}}}}}\text{ \ \ \ \ \ } \\ 
\text{\ \ \ \ \ \ \ } \\ 
\text{ \ \ \ \ }=a-1+\frac{1}{1+\frac{1}{a-1+a-1+\frac{1}{1+\frac{1}{2a-2+%
\frac{1}{...2a-2+1}}}}} \\ 
\\ 
=a-1+\frac{1}{1+\frac{1}{a-1+\frac{x_{n}}{y_{n}}}}\text{ \ \ \ \ \ \ \ \ \ \
\ \ \ \ \ \ } \\ 
\text{\ \ } \\ 
\text{ }=\frac{ax_{n}+\left( a^{2}-1\right) y_{n}}{ay_{n}+x_{n}}.\text{ \ \
\ \ \ \ \ \ \ \ \ \ \ \ \ \ \ \ \ \ \ \ \ }%
\end{array}%
\end{equation*}

$\left( x_{n+1},y_{n+1}\right) $ is a solution of $x^{2}-\left(
a^{2}-1\right) y^{2}=1$ since $x_{n+1}^{2}-\left( a^{2}-1\right)
y_{n+1}^{2}=\left( ax_{n}+\left( a^{2}-1\right) y_{n}\right) ^{2}-\left(
a^{2}-1\right) \left( ay_{n}+x_{n}\right) ^{2}=x_{n}^{2}-\left(
a^{2}-1\right) y_{n}^{2}=1.$

If $b>1$, then the proof is made by induction in a similar way.
\end{proof}
\end{theorem}

\begin{theorem}
All positive integer solutions of the equation $x^{2}-\left(
a^{2}b^{2}-b\right) y^{2}=1$ are given by%
\begin{equation*}
\left( x_{n},y_{n}\right) =\left\{ 
\begin{array}{c}
\left( \left( V_{n}\left( 2a,-1\right) \right) /2,U_{n}\left( 2a,-1\right)
\right) \\ 
\text{\ \ \ \ }\left( \left( V_{n}\left( 4a^{2}b-2,-1\right) \right)
/2,2aU_{n}\left( 4a^{2}b-2,-1\right) \right)%
\end{array}%
\right. ,%
\begin{array}{c}
\text{if }b=1 \\ 
\text{if }b>1%
\end{array}%
\end{equation*}%
with $n\geq 1$.
\end{theorem}

\begin{proof}
If $b>1$, then the followings hold:

By Theorem 4-ii, all positive integer solutions of the equation $%
x^{2}-\left( a^{2}b^{2}-b\right) y^{2}=1$ are given by%
\begin{equation*}
x_{n}+y_{n}\sqrt{a^{2}b^{2}-b}=\left( 2a^{2}b-1+2a\sqrt{a^{2}b^{2}-b}\right)
^{n}
\end{equation*}%
with $n\geq 1$. Assume that $\alpha =2a^{2}b-1+2a\sqrt{a^{2}b^{2}-b}$ and $%
\beta =2a^{2}b-1-2a\sqrt{a^{2}b^{2}-b}$. Then $\alpha -\beta =4a\sqrt{%
a^{2}b^{2}-b}$.%
\begin{equation*}
x_{n}+y_{n}\sqrt{a^{2}b^{2}-b}=\alpha ^{n}
\end{equation*}%
and%
\begin{equation*}
x_{n}-y_{n}\sqrt{a^{2}b^{2}-b}=\beta ^{n}.
\end{equation*}

Therefore $x_{n}=\frac{\alpha ^{n}+\beta ^{n}}{2}=\frac{V_{n}\left(
4a^{2}b-2,-1\right) }{2}$ and $y_{n}=\frac{\alpha ^{n}-\beta ^{n}}{2\sqrt{%
a^{2}b^{2}-b}}=2a\frac{\alpha ^{n}-\beta ^{n}}{\alpha -\beta }=2aU_{n}\left(
4a^{2}b-2,-1\right) .$ That is, $\left( x_{n},y_{n}\right) =\left( \frac{%
V_{n}\left( 4a^{2}b-2,-1\right) }{2},2aU_{n}\left( 4a^{2}b-2,-1\right)
\right) .$

Similarly it can be shown that if $b=1$, then $\left( x_{n},y_{n}\right)
=\left( \frac{V_{n}\left( 2a,-1\right) }{2},U_{n}\left( 2a,-1\right) \right)
.$
\end{proof}

\begin{theorem}
The Pell equation $x^{2}-\left( a^{2}b^{2}-b\right) y^{2}=-1$ has no
positive integer solutions.
\end{theorem}

\begin{proof}
The lenghts of the period of continued fraction $\sqrt{a^{2}b^{2}-b}$ are
even, then this equation has no positive integer solutions.
\end{proof}

\begin{theorem}
The fundamental solution of the Pell equation $x^{2}-\left(
a^{2}b^{2}-b\right) y^{2}=4$ is%
\begin{equation*}
\left( x_{1},y_{1}\right) =\left\{ 
\begin{array}{c}
\left( 2a,2\right) \\ 
\left( 4a^{2}b-2,4a\right) \text{\ \ \ \ }%
\end{array}%
\right. ,%
\begin{array}{c}
\text{if }b=1 \\ 
\text{if }b>1\text{.}%
\end{array}%
\end{equation*}
\end{theorem}

\begin{proof}
It is obvious from Theorem 3 and Theorem 4-ii.
\end{proof}

\begin{theorem}
All positive integer solutions of the equation $x^{2}-\left(
a^{2}b^{2}-b\right) y^{2}=4$ are given by%
\begin{equation*}
\left( x_{n},y_{n}\right) =\left\{ 
\begin{array}{c}
\left( V_{n}\left( 2a,-1\right) ,2U_{n}\left( 2a,-1\right) \right) \\ 
\text{\ \ \ \ }\left( V_{n}\left( 4a^{2}b-2,-1\right) ,4aU_{n}\left(
4a^{2}b-2,-1\right) \right)%
\end{array}%
\right. ,%
\begin{array}{c}
\text{if }b=1 \\ 
\text{if }b>1%
\end{array}%
\end{equation*}%
with $n\geq 1$.
\end{theorem}

\begin{proof}
It is trivial from Theorem 3 and Theorem 5.
\end{proof}

\begin{corollary}
All positive integer solutions of the equation $x^{2}-\left( 9k^{2}-3\right)
y^{2}=1$ are given by%
\begin{equation*}
\left( x_{n},y_{n}\right) =\ \left( \left( V_{n}\left( 12k^{2}-2,-1\right)
\right) /2,2kU_{n}\left( 12k^{2}-2,-1\right) \right)
\end{equation*}%
with $n\geq 1$.
\end{corollary}

\begin{corollary}
All positive integer solutions of the equation $x^{2}-\left( 9k^{2}-3\right)
y^{2}=4$ are given by%
\begin{equation*}
\left( x_{n},y_{n}\right) =\left( V_{n}\left( 12k^{2}-2,-1\right)
,4kU_{n}\left( 12k^{2}-2,-1\right) \right)
\end{equation*}%
with $n\geq 1$.
\end{corollary}

\bigskip \textbf{3. The Pell equation }$x^{2}-\left( a^{2}b^{2}-2b\right)
y^{2}=N$\textbf{\ }

\bigskip Now, we consider the integer solutions of the Pell equation%
\begin{equation}
F:x^{2}-\left( a^{2}b^{2}-2b\right) y^{2}=1.  \tag{2}
\end{equation}

\begin{theorem}
Let $F$ be the Pell equation in $\left( 2\right) $ and $a\geq 3$. Then the
followings hold:

$\left( \mathbf{i}\right) $ The continued fraction expansion of $\sqrt{%
a^{2}b^{2}-2b}$ is 
\begin{equation*}
\sqrt{a^{2}b^{2}-2b}=\left[ ab-1;\overline{1,a-2,1,2ab-2}\right] .
\end{equation*}

$\left( \mathbf{ii}\right) $ The fundamental solution is%
\begin{equation*}
\left( x_{1},y_{1}\right) =\left( a^{2}b-1,a\right) .
\end{equation*}

$\left( \mathbf{iii}\right) $ The n-th solution $\left( x_{n},y_{n}\right) $
can be find by%
\begin{equation*}
\frac{x_{n}}{y_{n}}=\left[ ab-1;\left( 1,a-2,1,2ab-2\right) _{n-1},1\right]
\end{equation*}%
where $\left( 1,a-2,1,2ab-2\right) _{n-1}$ means that there are $n-1$
successive terms $\left( 1,a-2,1,2ab-2\right) $.\qquad \qquad \qquad \qquad
\qquad

\begin{proof}
$\left( \mathbf{i}\right) $ 
\begin{equation*}
\begin{array}{c}
\sqrt{a^{2}b^{2}-2b}=ab-1+\left( \sqrt{a^{2}b^{2}-2b}-(ab-1)\right) \text{ \
\ \ \ \ \ \ \ \ \ \ \ \ \ \ \ \ \ \ \ \ \ \ \ \ \ \ \ \ \ \ \ \ \ \ \ \ \ \
\ \ \ \ \ \ \ \ \ \ \ \ \ \ \ \ \ } \\ 
\\ 
=ab-1+\frac{1}{\frac{\sqrt{a^{2}b^{2}-2b}+(ab-1)}{2ab-2b-1}}=ab-1+\frac{1}{1+%
\frac{\sqrt{a^{2}b^{2}-2b}-(ab-2b)}{2ab-2b-1}}\text{ \ \ \ \ \ \ \ \ \ \ \ \
\ \ } \\ 
\\ 
=ab-1+\frac{1}{1+\frac{1}{\frac{\sqrt{a^{2}b^{2}-2b}+\left( ab-2b\right) }{2b%
}}}=ab-1+\frac{1}{1+\frac{1}{a-2+\frac{\sqrt{a^{2}b^{2}-2b}-\left(
ab-2b\right) }{2b}}}\text{ \ \ } \\ 
\\ 
\text{ \ \ \ }=ab-1+\frac{1}{1+\frac{1}{a-2+\frac{1}{\frac{\sqrt{%
a^{2}b^{2}-2b}+\left( ab-2b\right) }{2ab-2b-1}}}}=ab-1+\frac{1}{1+\frac{1}{%
a-2+\frac{1}{1+\frac{1}{\frac{2ab-2b-1}{\sqrt{a^{2}b^{2}-2b}-\left(
ab-1\right) }}}}} \\ 
\\ 
\text{ \ \ \ \ \ \ \ \ \ \ \ \ \ \ \ }=ab-1+\frac{1}{1+\frac{1}{a-2+\frac{1}{%
1+\frac{1}{\sqrt{a^{2}b^{2}-2b}+\left( ab-1\right) }}}}=ab-1+\frac{1}{1+%
\frac{1}{a-2+\frac{1}{1+\frac{1}{2ab-2+\left( \sqrt{a^{2}b^{2}-2b}-\left(
ab-1\right) \right) }}}}.%
\end{array}%
\end{equation*}%
This completes the proof.

$\left( \mathbf{ii}\right) $ From $x_{n}=p_{nm-1}$ and $y_{n}=q_{nm-1}$, we
obtain $x_{1}=p_{3}$ and $y_{1}=q_{3}$. Therefore we must find $p_{3}$ and $%
q_{3}$. Now, we can find with the help of $3th$ convergent of $\sqrt{%
a^{2}b^{2}-2b}$.

$\frac{p_{3}}{q_{3}}=\left[ ab-1;1,a-2,1\right] =ab-1+\frac{1}{1+\frac{1}{%
a-2+\frac{1}{1}}}=\frac{a^{2}b-1}{a}$. Therefore $\left( x_{1},y_{1}\right)
=\left( a^{2}b-1,a\right) $.

$\left( \mathbf{iii}\right) $ The proof is made by a similar manner as in
the proof of the Theorem 4-iii.
\end{proof}
\end{theorem}

\begin{theorem}
All positive integer solutions of the equation $x^{2}-\left(
a^{2}b^{2}-2b\right) y^{2}=1$ are given by%
\begin{equation*}
\left( x_{n},y_{n}\right) =\left( \left( V_{n}\left( 2a^{2}b-2,-1\right)
\right) /2,aU_{n}\left( 2a^{2}b-2,-1\right) \right)
\end{equation*}%
with $n\geq 1$.
\end{theorem}

\begin{proof}
By Theorem 9-ii, all positive integer solutions of the equation $%
x^{2}-\left( a^{2}b^{2}-2b\right) y^{2}=1$ are given by%
\begin{equation*}
x_{n}+y_{n}\sqrt{a^{2}b^{2}-2b}=\left( a^{2}b-1+a\sqrt{a^{2}b^{2}-2b}\right)
^{n}
\end{equation*}%
with $n\geq 1$. Assume that $\alpha =a^{2}b-1+a\sqrt{a^{2}b^{2}-2b}$ and $%
\beta =a^{2}b-1-a\sqrt{a^{2}b^{2}-2b}$. Then $\alpha -\beta =2a\sqrt{%
a^{2}b^{2}-2b}$.%
\begin{equation*}
x_{n}+y_{n}\sqrt{a^{2}b^{2}-2b}=\alpha ^{n}
\end{equation*}%
and%
\begin{equation*}
x_{n}-y_{n}\sqrt{a^{2}b^{2}-2b}=\beta ^{n}.
\end{equation*}

Therefore $x_{n}=\frac{\alpha ^{n}+\beta ^{n}}{2}=\frac{V_{n}\left(
2a^{2}b-2,-1\right) }{2}$ and $y_{n}=\frac{\alpha ^{n}-\beta ^{n}}{2\sqrt{%
a^{2}b^{2}-2b}}=a\frac{\alpha ^{n}-\beta ^{n}}{\alpha -\beta }=aU_{n}\left(
2a^{2}b-2,-1\right) .$ That is, $\left( x_{n},y_{n}\right) =\left( \frac{%
V_{n}\left( 2a^{2}b-2,-1\right) }{2},aU_{n}\left( 2a^{2}b-2,-1\right)
\right) .$
\end{proof}

\begin{theorem}
The Pell equation $x^{2}-\left( a^{2}b^{2}-2b\right) y^{2}=-1$ has no
positive integer solutions.
\end{theorem}

\begin{proof}
The lenghts of the period of continued fraction $\sqrt{a^{2}b^{2}-2b}$ is
even, then this equation has no positive integer solutions.
\end{proof}

\begin{theorem}
The fundamental solution of the Pell equation $x^{2}-\left(
a^{2}b^{2}-2b\right) y^{2}=4$ is%
\begin{equation*}
\left( x_{1},y_{1}\right) =\left( 2a^{2}b-2,2a\right) .
\end{equation*}
\end{theorem}

\begin{proof}
It is clear from Theorem 3 and Theorem 9-ii.
\end{proof}

\begin{theorem}
All positive integer solutions of the equation $x^{2}-\left(
a^{2}b^{2}-2b\right) y^{2}=4$ are given by%
\begin{equation*}
\left( x_{n},y_{n}\right) =\left( V_{n}\left( 2a^{2}b-2,-1\right)
,2aU_{n}\left( 2a^{2}b-2,-1\right) \right)
\end{equation*}%
with $n\geq 1$.
\end{theorem}

\begin{proof}
The proof can be easily seen from Theorem 3 and Theorem 10.
\end{proof}

\begin{theorem}
The Pell equation $x^{2}-\left( a^{2}b^{2}-2b\right) y^{2}=-4$ has no
positive integer solutions.
\end{theorem}

\begin{proof}
Let $b$ be odd. If $a$ is odd, then $a^{2}b^{2}-2b\equiv 3\left( \func{mod}%
4\right) $. If $a$ is even, then $a^{2}b^{2}-2b\equiv 2\left( \func{mod}%
4\right) $. From Theorem 1, we know that the equation $x^{2}-\left(
a^{2}b^{2}-2b\right) y^{2}=-4$ has positive integer solutions if and only if
the equation $x^{2}-\left( a^{2}b^{2}-2b\right) y^{2}=-1$ has positive
integer solutions. But from Theorem 11, the equation $x^{2}-\left(
a^{2}b^{2}-2b\right) y^{2}=-1$ has no positive integer solutions. Therefore,
the equation $x^{2}-\left( a^{2}b^{2}-2b\right) y^{2}=-4$ has no positive
integer solutions.

Let $b$ be even. Then $a^{2}b^{2}-2b$ is even. Assume by way of
contradiction that there are positive integers $m$ and $n$ such that $%
m^{2}-\left( a^{2}b^{2}-2b\right) n^{2}=-4$. Both $b$ and $a^{2}b^{2}-2b$
are even. Therefore, $m$ is even. Let $b=2t$ where $t\in 
\mathbb{Z}
^{+}$. Then $m^{2}-\left( a^{2}4t^{2}-4t\right) n^{2}=-4$ and we get $\left(
m/2\right) ^{2}-\left( a^{2}t^{2}-t\right) n^{2}=-1$. We know from Theorem
6, the equation $x^{2}-\left( a^{2}b^{2}-b\right) y^{2}=-1$ has no positive
integer solutions. So this is a contradiction. Then the Pell equation $%
x^{2}-\left( a^{2}b^{2}-2b\right) y^{2}=-4$ has no positive integer
solutions.
\end{proof}

\begin{corollary}
All positive integer solutions of the equation $x^{2}-\left( 9k^{2}-6\right)
y^{2}=1$ are given by%
\begin{equation*}
\left( x_{n},y_{n}\right) =\left( \left( V_{n}\left( 6k^{2}-2,-1\right)
\right) /2,kU_{n}\left( 6k^{2}-2,-1\right) \right)
\end{equation*}%
with $n\geq 1$.
\end{corollary}

\begin{corollary}
All positive integer solutions of the equation $x^{2}-\left( 9k^{2}-6\right)
y^{2}=4$ are given by%
\begin{equation*}
\left( x_{n},y_{n}\right) =\left( V_{n}\left( 6k^{2}-2,-1\right)
,2kU_{n}\left( 6k^{2}-2,-1\right) \right)
\end{equation*}%
with $n\geq 1$.
\end{corollary}

\textbf{Acknowledgement. }\textit{This research is supported by TUBITAK (The
Scientific and Technological Research Council of Turkey) and Necmettin
Erbakan University Scientific Research Project Coordinatorship (BAP). This
study is a part of the corresponding author's Ph.D. Thesis.}

\end{document}